\documentclass[letterpaper, 10pt]{amsart}

\usepackage{mymacros}

\begin{document}

\title{On Geometry and Topology of 4-Orbifolds}
\author{Dmytro Yeroshkin}
\address{Syracuse University, Department of Mathematics, 215 Carnegie Building, Syracuse, NY 13244, USA}
\email{dyeroshk@syr.edu}

\begin{abstract}
We prove an analogue of the result of Hsiang and Kleiner for 4-dimensional compact orbifolds with positive curvature and an isometric $S^1$ action. Additionally, we prove that when $\pi_1(|\calO^4|)=0$, then $\pi_1^{orb}(\calO)$ provides a bound on the failure of $\bbZ$-valued Poincar\'e Duality of $|\calO|$, and if $\pi_1^{orb}(\calO)=0$, then $\bbZ$-valued Poincr\'e Duality holds for $|\calO|$.
\end{abstract}
\maketitle

\section{Introduction}\label{SecIntro}

One important question in Riemannian geometry is what spaces admit metrics of positive curvature. In particular, the results that distinguish between manifolds admitting non-negative curvature and those admitting positive curvature are the theorems of Bonnet-Myers and Synge in the compact case and Perelman's proof of the soul conjecture in the non-compact case. If we add assumptions on the size of the isometry group, then we have the result of Hsiang and Kleiner \cite{HK4Dim}, that a positively curved 4-dimensional Riemannian manifold with an isometric $S^1$ action is homeomorphic to either $S^4,\RP^4$ or $\CP^2$ (in fact by results of Fintushel \cite{FCircle} this is true up to diffeomorphism). In higher dimensions, the assumption of a larger isometry group can be used, see \cite{GSSym}, \cite{WTorAct}, \cite{FRHomPC}, \cite{KHopf} and \cite{WPCM}. More recently, some work has been done on this question in a more general setting, see the work of Harvey and Searle \cite{HSAlex} and Galaz-Garcia and Guijarro \cite{GGGAlex} for results on positively curved Alexandrov spaces. In this paper, we prove an analogue of the result of Hsiang and Kleiner in the case of orbifolds.

\begin{mthm}\label{ThmMain}
Let $\calO$ be a compact 4-dimensional, positively curved Riemannian orbifold with an isometric $S^1$ action and $\pi_1^{orb}(\calO) = 0$, then one of the following holds:

\begin{enumerate}
\item either $|\calO|$ is homotopy equivalent to $S^4$,

\item or $H^*(|\calO|;\bbZ)=H^*(\CP^2;\bbZ)$.
\end{enumerate}

Furthermore, if the $S^1$ action has a 2-dimensional fixed point set, then

\begin{enumerate}
\item either $|\calO|$ is homeomorphic to $S^4$,

\item or $|\calO|$ is homeomorphic to the underlying space of $\CP^2[\lambda_0,\lambda_1,\lambda_2]$ for some positive integers $\lambda_0, \lambda_1, \lambda_2$.
\end{enumerate}

In both of these cases, the $S^1$ action is equivariant to a linear action.
\end{mthm}

Here we denote by $|\calO|$ the underlying topological space of an orbifold $\calO$. Recall that weighted projective spaces, denoted by $\CP^2[\lambda_0,\lambda_1,\lambda_2]$, $\lambda_i\in\bbZ^+$ (or $\CP^2[\lambda]$ for short), are $4$-dimensional orbifolds which can be written as $S^5/S^1_\lambda$, where $S^5\subset\bbC^3$ and $z\in S^1_\lambda$ acts on $(w_0,w_1,w_2)\in\bbC^3$ as $(z^{\lambda_0}w_0,z^{\lambda_1}w_1,z^{\lambda_2}w_2)$. Taking the round metric on $S^5$, we get a natural metric with positive sectional curvature on $\CP^2[\lambda_0,\lambda_1,\lambda_2]$, for which we still have an isometric $S^1$ action.

If $\pi_1^{orb}(\calO)\neq 0$, then the orbifold is a finite quotient of one of the cases listed above. It is worth noting, that in \cite{HSAlex} the authors claim a similar result for 4-dimensional Alexandrov spaces.

In the work of Hsiang-Kleiner, the authors use the work of Freedman \cite{F4M} to provide the topological classification. Since no such work exists for orbifolds, we use the work of Perelman on Alexandrov spaces \cite{PAlex} for the case when the fixed point set has dimension two. For the case when the fixed point set consists of isolated points, we classify the cohomology of $|\calO|$.

\begin{mthm}\label{ThmTop}
Let $\calO$ be a compact, orientable, 4-dimensional orbifold with $\pi_1(|\calO|)=0$, and $\chi(|\calO|)=n$. Then,
\[
H^k(|\calO|;\bbZ) = \begin{cases}
\bbZ, & k=0,4\\
0, & k=1\\
\bbZ^{n-2}, & k=2\\
\tau,&k=3,
\end{cases}
\]
where $\tau$ is some torsion group such that there is a surjective map $\phi:\pi_1^{orb}(\calO)\to \tau = H^3(|\calO|)$.

In particular, if $\pi_1^{orb}(\calO)=0$, then
\[
H^k(|\calO|;\bbZ) =\begin{cases}
\bbZ,&k=0,4\\
0,&k=1,3\\
\bbZ^{n-2},&k=2.
\end{cases}
\]
\end{mthm}

So, we can apply this theorem to better understand a family of orbifolds introduced by the author in \cite{YSU3}:

\begin{cor}
Every $S^1$ quotient of the Wu manifold ($S^1_{p,q}\backslash SU(3)/SO(3)$) has integer cohomology of $\CP^2$.
\end{cor}

A portion of this work was part of the author's Ph.D. thesis. The author would like to thank his advisor, Wolfgang Ziller, for his invaluable advice and encouragements.

\section{Preliminaries on Orbifolds}\label{SecOrbi}

Recall that an $n$-dimensional orbifold $\calO^n$ is a space modeled locally on $\bbR^n/\Gamma$ with $\Gamma\subset O(n)$ finite. Given a point $p\in\calO$, the orbifold group at $p$, which we'll denote as $\Gamma_p$ is the subgroup of the group $\Gamma$ in the local chart $\bbR^n/\Gamma$, that fixes a lift of $p$ to $\bbR^n$. Note that different choices of a lift of $p$ result in $\Gamma_p$ being conjugated, and as such, we will think of $\Gamma_p$ up to conjugacy.

In many of the results in this paper, it will be useful to think of an orbifold $\calO^n$ as a disjoint collection of connected strata. Each stratum, denoted by $\calS$ is a connected component of points with the same (up to conjugacy) orbifold group. The stratum containing $p$ will often be denoted by $\calS(p)$. One stratum deserves special mention, $\calO^{reg}$ is the stratum of points with trivial orbifold groups, and points in $\calO^{reg}$ are called regular. Furthermore $\calO^{reg}$ is an open dense subset of $\calO$.

Recall that an orbifold $\calO$ is said to be orientable if $\calO^{reg}$ is orientable, and each orbifold group $\Gamma_p$ preserves orientation, that is $\Gamma_p\subset SO(n)$ for each $p\in\calO$. A choice of an orientation on $\calO$ is a choice of an orientation on $\calO^{reg}$.

We call an orbifold $\calU$ a cover of $\calO$ if $\calO = \calU/\Gamma$, with $\Gamma$ discrete, such that the action of $\Gamma$ preserves the orbifold structure. We recall the definition of $\pi_1^{orb}$, the orbifold fundamental group. If $\calO=\calU/\Gamma$, with $\Gamma$ discrete and $\calU$ admitting no covers, then $\pi_1^{orb}(\calO)=\Gamma$.

\begin{rmk}
There is an alternative definition of an orbifold fundamental group, which is analogous to the path homotopy definition of the topological fundamental group. For precise definition and proof of the equivalence see \cite{SG3M}.
\end{rmk}

The proposition below demonstrates the analog of Synge's Theorem for Orbifolds, a version for Alexandrov spaces can be found in \cite{HSAlex}, and in with the additional assumption of local orientability in the even dimensional case in \cite{PPar}.

\begin{prop}[Synge's Theorem for Orbifolds]\label{PropSynge}
Let $\calO$ be a compact positively curved orbifold, then

\begin{enumerate}
\item if $n$ is even, and $\calO$ orientable, then $|\calO|$ is simply connected.

\item if $n$ is odd, and for every $p\in\calO$, $\Gamma_p\subset SO(n)$, then $\calO$ is orientable.
\end{enumerate}
\end{prop}

We add in this section a few warnings about situations in which the behavior of orbifolds differs significantly from that of manifolds.

If a finite group fixes two complementary subspaces of $T_p\calO$, the tangent cone at $p$, which is defined as the space of directions (with norm), or alternatively as the quotient of $T_{\bar{p}}\tilde{\calU}/\Gamma_p$ of the tangent space in the local cover, then it need not fix the entire tangent space.

Additionally, there is a challenge in defining when two vectors in $T_p\calO$ are orthogonal. The first possible definition is $v$ and $w$ are orthogonal if there are lifts $\tilde{v},\tilde{w}$ which are orthogonal in the local cover. However, with this definition, it is possible to have $v$ orthogonal to itself, for example $\calO=S^2/rot_{\pi/2}$ then at the north pole, every vector is orthogonal to itself. The second possible definition is to say $v,w$ are orthogonal if all lifts $\tilde{v},\tilde{w}$ are orthogonal in the local cover. Under this definition, it is possible that $v$ has no non-zero vectors orthogonal to it (same $\calO$ as before); nevertheless, this is the definition we prefer.

We now observe that the same proof as for manifolds shows that the Slice Theorem holds for orbifolds. We observe that a Slice Theorem holds for Alexandrov spaces \cite{HSAlex} as well, but the proof is significantly more difficult and does not consider how it interacts with the orbifold structure.

\begin{prop}[Slice Theorem for Orbifolds]
Let $G$ be a compact connected Lie group, $\calO^n$ a Riemannian orbifold, with an isometric $G$-action. Then, given any $p\in\calO$, and sufficiently small $r>0$, we have a $G$-equivariant orbifold diffeomorphism
\[
B_r(G(p))\cong G\times_{G_p} Cone(\nu_p),
\]
where $\nu_p=\{v\in T_p\calO|v\perp G(p),|v|=1\}$ is the space of directions orthogonal to the orbit.
\end{prop}

\begin{rmk}
In the course of the proof we will also see that it does not matter which definition of orthogonality we use.
\end{rmk}

\begin{proof}
We begin with the observation that $G(p)\subset\calS(p)$, since $g(p)$ must lie in a stratum with the same orbifold group, and $G(p)$ is connected. In particular, this tells us that $G(p)$ lifts uniquely to the local manifold cover at $p$, since $\calS(p)$ must lift to $\Fix(\Gamma_p)$ in the local cover. Furthermore, with respect to this lift, $\Gamma_p\subset O(k)\subset O(n)$, where $k$ is the codimension of $\calS(p)$. Let $l$ be the codimension of $G(p)$ inside $\calS(p)$, then $\nu_p = S^{k+l-1}/\Gamma_p$, with $\Gamma_p$ fixing $\bbR^l$. This is defined independent of our choice of definition of orthogonality, since any direction along $G(p)$ lifts uniquely.

These observations allow us to approach the proof for the orbifold case in the same fashion as the manifold case. We define $\phi:G\times Cone(\nu_p)\to\calO$ as
\[
\phi(g,v)=g(\exp_p(v)),
\]
this map has fiber $G_p$, and so induces a map $G\times_{G_p} Cone(\nu_/p) \to\calO$. To get an inverse map, for $q$ close to $G(p)$, we take $q_0$ to be the point on $G(p)$ closest to $q$, and let $g_0\in G$ be such that $g_0(p)=q_0$. This choice is unique up to elements of $G_p$. We then consider $v_0\in\Cone(\nu_p)\subset T_p\calO$ such that $\exp_p(v_0)=g_0^{-1}(q)$. Clearly, $\phi(g_0,v_0)=q$.
\end{proof}

We end this section with a result of Satake \cite{SVM}, which is the first paper that introduced orbifolds (although the initial terminology used in that paper is V-manifolds).

\begin{prop}[Orbifold Poincar\'e Duality (\cite{SVM})]\label{PropOPD}
Let $\calO^n$ be an orientable compact orbifold, then
\[
H_k(|\calO|;\bbR) = H^{n-k}(|\calO|;\bbR).
\]
\end{prop}

The proof relies on the fact that if $\Gamma\subset SO(n)$ finite, then $H^*(S^{n-1}/\Gamma;\bbR) = H^*(S^{n-1};\bbR)$ \cite{GrTohoku}.

\section{Examples}\label{SecEx}

In this section we provide some examples of 4-dimensional orbifolds with isometric $S^1$ actions. We also see that there can be many such orbifolds with the same underlying space, but different singular structures. We also provide an example of a family of 4-orbifolds where $\bbZ$-valued Poincar\'e Duality does not hold.

\begin{ex}[Weighted Projective Spaces]\label{ExWPS}
Let $\lambda_0,\lambda_1,\lambda_2$ be positive integers such that
\[
\gcd(\lambda_0,\lambda_1,\lambda_2)=1.
\]
We define an $S^1$ action on $S^5\subset\bbC^3$ by
\[
z\star(w_0,w_1,w_2) = (z^{\lambda_0}w_0,z^{\lambda_1}w_1,z^{\lambda_2}w_2).
\]
The quotient space $S^5/S^1$ is an orientable 4-dimensional orbifold, known as a weighted projective space, denoted $\CP^2[\lambda_0,\lambda_1,\lambda_2]$, or $\CP^2[\lambda]$ for short.

As with smooth projective spaces, we use homogeneous coordinates, i.e $[w_0:w_1:w_2]$ denotes the orbit of $(w_0,w_1,w_2)$.

A typical question when studying orbifolds is: what is the orbifold structure of this space? i.e. what are the singular points, and what are the corresponding orbifold groups?

For $\CP^2[\lambda]$, the following are all the possible non-trivial orbifold groups:
\begin{align*}
\Gamma_{[1:0:0]}&=\bbZ_{\lambda_0}&
\Gamma_{[0:1:0]}&=\bbZ_{\lambda_1}&
\Gamma_{[0:0:1]}&=\bbZ_{\lambda_2}\\
\Gamma_{[w_0:w_1:0]}&=\bbZ_{(\lambda_0,\lambda_1)}&
\Gamma_{[w_0:0:w_2]}&=\bbZ_{(\lambda_0,\lambda_2)}&
\Gamma_{[0:w_1:w_2]}&=\bbZ_{(\lambda_1,\lambda_2)}.
\end{align*}
As we can see, $\CP^2[\lambda]$ has a singular set consisting of up to three points corresponding to $[1:0:0],[0:1:0],[0:0:1]$ and of up to three (possibly singular) $S^2$'s connecting pairs of such points, which correspond to $[w_0:w_1:0],[w_0:0:w_2],[0:w_1:w_2]$.

The metric on $\CP^2[\lambda]$ induced by the round metric on $S^5$ has positive sectional curvature. Furthermore, the natural action by $T^3$
\[
(z_0,z_1,z_2)\star[w_0:w_1:w_2]=[z_0w_0:z_1w_1:z_2w_2]
\]
has ineffective kernel $S^1=\{(z^{\lambda_0},z^{\lambda_1},z^{\lambda_2})\}$, and hence induces an isometric $T^2$ action on $\CP^2[\lambda]$.

One can now ask what the stratification of $\CP^2[\lambda]$ is that is induced by an $S^1$ action, where $S^1\subset T^2$. The fixed point set of an $S^1$ action on $\CP^2[\lambda]$ can be either three isolated points, or an isolated point and a (possibly singular) $S^2$. The former corresponds to a generic $S^1$ action, and the fixed points are precisely $[1:0:0],[0:1:0],[0:0:1]$. The latter case has $[1:0:0],[0:w_1:w_2]$ (or similar pairs) as its fixed point set, and corresponds to $S^1$ actions that can be written as $(z,1,1)\in T^3$.

We finish this example by observing that since $\CP^2[\lambda] = S^5/S^1$, the exact homotopy sequence for orbifolds fibrations implies that $\pi_1^{orb}(\CP^2[\lambda])=0$. Also, using Mayer-Vietoris, one can observe that $H^*(|\CP^2[\lambda]|;\bbZ)=H^*(\CP^2;\bbZ)$. Furthermore, $|\CP^2[\lambda]|=\CP^2$ iff $\lambda_0=ab,\lambda_1=ac,\lambda_2=bc$.
\end{ex}

\begin{ex}\label{ExL5}
Consider $\calO=\CP^2[1,2,4]$. and let $S^1$ act on it by
\[
z\star[w_0:w_1:w_2]=[zw_0:w_1:w_2] = [w_0:\bar{z}^2 w_1: \bar{z}^4 w_2],
\]
which has ineffective kernel $z=\pm1$.

The fixed point set consists of an isolated point: $[1:0:0]$, and a singular $S^2$: $\{[0:w_1:w_2]\}$. To clearly see the representation of $S^1$ on a neighborhood of $[1:0:0]$, we re-write this action in an effective way as
\[
u\star[w_0:w_1:w_2]=[w_0:uw_1:u^2w_2],
\]
where one can think of $u$ as $\bar{z}^2$. Since the tangent space at $[1:0:0]$ is spanned by $(0,z,w)$, we observe that the action of $S^1$ on this space has the isotropy representation equivalent to $\phi_{1,2}$, where $\phi_{k,l}$ is the action of $S^1$ on $\bbC^2=\bbR^4$ given by $S^1 = \{(z^k,z^l)\} \subset T^2$.
\end{ex}

This example demonstrates something that can not happen in the manifold case, since Hsiang and Kleiner (Lemma 5 in \cite{HK4Dim}) show that if the fixed point set contains an isolated point and a 2-dimensional component, then the isotropy representation of $S^1$ on a neighborhood of the isolated point has to be $\phi_{1,1}$. In particular, in this case the proof of \cite{HK4Dim} can not immediately be generalized to orbifolds.

\begin{ex}\label{ExHitch}
Another interesting family of examples are the Hitchin family of orbifolds introduced in \cite{HEinst}. Recall that a Hitchin orbifold, which we will denote $H_k$, has $S^4$ as its underlying space, and its singular locus consists of a smooth Veronese $\RP^2$ with a $\bbZ_k$ orbifold group. In particular, we view $S^4$ as the set of traceless symmetric 3x3 matrices with unit norm, on which $SO(3)$ acts by conjugation. We can view each orbit as the space of matrices with fixed eigenvalues.

The singular orbits of the $SO(3)$ action are precisely two copies of $\RP^2$, corresponding to the matrices with repeated positive or negative eigenvalues. To construct the Hitchin $k$-orbifold $H_k$, we introduce a $\bbZ_k$ singularity along one of the $\RP^2$ orbits. One way of interpreting this is to replace the existing $D^2$ bundle over $\RP^2$ by a $D^2/\bbZ_k$ cone-bundle over $\RP^2$.

Next, consider the action of $SO(3)$ on $\CP^2$ induced by the canonical embedding $SO(3)\subset SU(3)$. Recall that there exists a branched cover $\CP^2\to S^4$ where we identify $[w_0:w_1:w_2]$ with $[\bar{w_0}:\bar{w_1},\bar{w_2}]$, and this cover is an $SO(3)$-equivariant continuous map. If we impose a $\bbZ_k$ singularity along $\RP^2$, we obtain the universal cover of $H_{2k}$ ($\bar{H_{2k}}$).

The Hitchin metric on $H_k$ is self-dual Einstein, but has some negative curvature unless $k=1,2$, see \cite{ZCoho}, where $H_1$ is the standard $S^4$, and $H_2 = \CP^2/\bbZ_2$, where $\bbZ_2$ acts by conjugation. Furthemore, one can view $H_3 = S^7/SU(2)$, where $SU(2)$ acts by the irreducible representation on $\bbC^4\supset S^7$.

We note that a Hitchin orbifold has $\pi_1^{orb}(H_k)=0$ when $k$ is odd. When $k$ is even, $\pi_1^{orb}(H_k)=\bbZ_2$ and $H_k$ is double covered by $\CP^2$ with a singular $\RP^2$ where the orbifold group is $\bbZ_{k/2}$, with the cover given by the map $\CP^2\to S^4$ given by identifying $[z]$ with $[\bar{z}]$, where the branching locus is $\RP^2$.

Hitchin orbifolds are of interest in particular because the two infinite families of 7-dimensional candidates for cohomogeneity one manifolds with positive curvature ($P_k,Q_k$) can be described as bundles over the Hitchin orbifolds up to covers; namely, $S^3\to P_k\to H_{2k-1}$ and $S^3\to Q_k\to \bar{H_{2k}}$ (see \cite{GWZCoho} for the general construction, \cite{GVZPosC}, \cite{DP2} for positive curvature on $P_2$, and \cite{ZCoho} for an overview). It is conjectured that all manifolds $P_k,Q_k$ admit positive curvature (see \cite{ZExam}). Also, it is known that all $P_k,Q_k$ admit non-negative sectional curvature, hence so do all $H_k$.

As has been remarked in \cite{GVZPosC}, one can use Cheeger deformation to obtain a metric with positive sectional curvature on $H_k$. To do this carefully, one must utilize the work of M\"uter \cite{MuTh} to study the assymptotic behavior of the deformation.

Any circle $S^1\subset SO(3)$ still acts by isometries. Specifically let
\[
S^1 = \left\{\begin{pmatrix}
\cos t & -\sin t & 0\\
\sin t & \cos t & 0\\
0 & 0 & 1
\end{pmatrix} :t\in[0,2\pi)\right\}\subset SO(3).
\]
Since the $SO(3)$ action on each singular orbit is the standard $SO(3)$ action on $\RP^2$, the $S^1$ fixes two points, one in each of the singular orbits of the $SO(3)$ action. In particular, it fixes
\[
\begin{pmatrix}
\frac{1}{\sqrt{6}} &&\\
& \frac{1}{\sqrt{6}} &\\
&& \frac{-2}{\sqrt{6}}
\end{pmatrix}\;and\;
\begin{pmatrix}
\frac{-1}{\sqrt{6}} &&\\
& \frac{-1}{\sqrt{6}} &\\
&& \frac{2}{\sqrt{6}}
\end{pmatrix}.
\]
We can view the $S^1$ action as a suspension of an $S^1$ action on $S^3$. Indeed, if we view traceless 3x3 symmetric matrices in $S^4$ as
\[
\begin{pmatrix}
A & v\\
v^T & h
\end{pmatrix}\qquad\text{with } A=\begin{pmatrix}
-h/2+t & b\\
b & -h/2-t
\end{pmatrix},
\qquad\tr A + h = 0.
\]
Here $v=\begin{pmatrix} c \\ d \end{pmatrix}$ is a vector in $\bbR^2$, and $h$ is the suspension parameter. Observe that
\[
t^2+b^2+c^2+d^2=\frac{2-3h^2}{4}\qquad\text{and hence}\; h\in[-2/\sqrt{6},2/\sqrt{6}].
\]
Thus, we have a 3-sphere when $h\in(-2/\sqrt{6},2/\sqrt{6})$, and $t^2+b^2+c^2+d^2=0$ when $h=\pm2/\sqrt{6}$, so the sphere collapses to a point.

Suppose that the singular locus is the $\RP^2$ correpsonding to the matrices with eigenvalues $1/\sqrt{6}$, $1/\sqrt{6}$ and $-2/\sqrt{6}$. Then, conjugating $\diag(1/\sqrt{6},1/\sqrt{6},-2/\sqrt{6})$ we can see that this $\RP^2$ intersects only the spheres with $h\in[-2/\sqrt{6},1/\sqrt{6}]$. This intersection is precisely one orbit of the $S^1$ action (which acts as $\phi_{1,2}$ on the $S^3$'s, which can be seen from the $S^1$ action on $A,v$), and as we approach the last $S^3$ where the intersection is non-empty, this $S^1$ turns into the singular orbit.
\end{ex}

\begin{ex}
Let $L^3(p;q)$ be a 3-dimensional lens space $S^3/\bbZ_p$ with $[n]$ acting by multiplication by $(z^n,z^{qn})$ on $\bbC^2\supset S^3$, where $z$ is a primitive $p^{th}$ root of unity.

Let $\calL^4_{p;q}$ be the suspension of $L^3(p;q)$, which can be viewed as $S^4/\bbZ_p$, with $S^4\subset \bbC^2\oplus\bbR$ and $\bbZ_p$ acting trivially on the last coordinate. Then, $\pi_1^{orb}(\calL^4_{p;q})=\bbZ_p$, and
\[
H^k(|\calL^4_{p;q}|;\bbZ) = \begin{cases}
\bbZ,&k=0,4\\
0,&k=1,2\\
\bbZ_p,&k=3,
\end{cases}\hspace{2cm}
H_k(|\calL^4_{p;q}|;\bbZ) = \begin{cases}
\bbZ,&k=0,4\\
0,&k=1,3\\
\bbZ_p,&k=2.
\end{cases}
\]
Therefore, $\bbZ$-valued Poincar\'e Duality does not hold, but $\pi_1^{orb}(\calL^4_{p;q})$ bounds the size of the torsion group as claimed in Theorem \ref{ThmTop}.
\end{ex}

\section{General Structure}\label{SecStruct}

In this section, we focus specifically on the structure of 4-dimensional orbifolds with isometric $S^1$ action. We generally make no assumptions about the curvature.

\begin{lem}[Kobayashi]\label{LemKob}
Let $\calO^n$ be a compact Riemannian orbifold with an isometric $S^1$ action. Let $\calF$ be the suborbifold of fixed points of this action. Then,

\begin{enumerate}
\item Each connected component of $\calF$ is a totally geodesic suborbifold of even codimension.

\item $\chi(|\calF|)=\chi(|\calO|)$.
\end{enumerate}
\end{lem}

\begin{proof}
Part 1 is proved completely analogously to the proof for manifold case (See pp 59-61 of \cite{KTrans}), and follows from Representation Theory.

For part 2, we note that Kobayashi's original proof in \cite{KFix} only requires compactness to guarantee that the fixed point set can not be ``dense'' i.e. there exists $\epsilon>0$ such that $\epsilon$ neighborhoods of connected components of the fixed point set are disjoint. This condition is satisfied when we consider $S^1$ action on strata of a compact orbifold, since such strata are bounded, and the boundary of each stratum lies inside the orbifold. As such, what we have is $\chi(|\calF\cap\calS|) = \chi(|\calS|)$ for each stratum $\calS\subset\calO$. Gluing the $\calF\cap\calS$ pieces together we get part 2.
\end{proof}

As a consequence of this, we get

\begin{lem}
Let $\calO$ be a compact positively curved 4-dimensional Riemannian orbifold with a non-trivial isometric $S^1$-action, and $\calF$ be the set of fixed points of the action. Then, $\calF$ is non-empty and either consists of 2 or more isolated points, or has at least one 2-dimensional component.
\end{lem}

The proof is identical to that for manifolds, the only challenge is to verify that if $\calF$ consists of only isolated points, then $|\calF|\geq 2$. To do this we utilize Orbifold Poincar\'e Duality \ref{PropOPD}, and Synge's theorem \ref{PropSynge}

We also have additional structural restrictions on 4-dimensional orbifolds with
an isometric $S^1$ action.

\begin{prop}\label{OrbiGrp}
Let $\calO$ be a 4-dimensional Riemannian $S^1$-orbifold, and let $x\in\calO$ be a singular point. Then, we have either

\begin{align*}
\Gamma_x&\subset U(2)\subset SO(4),&&\textit{if x is a fixed point, or}\\
\Gamma_x&\subset SO(3)\subset SO(4),&&\textit{otherwise}.
\end{align*}

In particular, a neighborhood of $x$ is homeomorphic to a disk unless $x$ is a fixed point.
\end{prop}

\begin{proof}
Suppose $x$ is a fixed point. Consider the action of $S^1$ on $\bbR^4/\Gamma_x = T_x\calO$. Let $V$ denote the vector field associated to the $S^1$ action on $T_x\calO$, and $\tilde{V}$ its lift to $\bbR^4$.

$\tilde{V}$ is a vector field associated to an action of $\bbR$ on $\bbR^4$, furthermore, $t_0=2\pi$ acts as an element of $\Gamma_x$. Therefore, $T=2\pi n$ acts trivially for some $n\in\bbZ^+$, so the action is an $S^1$ action.

Up to conjugation, this $S^1$ action must be equal to $(e^{lti},e^{mti})\in T^2\subset S^3\times S^3 = Spin(4)$ (here we consider $Spin(4)$ instead of $SO(4)$ for convenience). Also, this $S^1$ must normalize $\Gamma_x$, and so must commute with $\Gamma_x$. This leaves us two cases, either $m\neq 0\neq l$ or one is zero. If neither $m$ nor $l$ is zero, then $\Gamma_x$ lifts to a subgroup of $T^2$, so $\Gamma_x\subset T^2$. Otherwise, the lift of $\Gamma_x$ lies in either $S^3\times S^1$ or $S^1\times S^3$, in both cases, we get $\Gamma_x\subset U(2)\subset SO(4)$.

Suppose $x$ is not a fixed point, then $\Gamma_x$ must fix at least one direction (along the orbit $S^1(x)$), so $\Gamma_x\subset SO(3)\subset SO(4)$.
\end{proof}

Since the orbifold group along a stratum must be constant (up to conjugacy), we conclude that

\begin{cor}
Let $\calS\subset\calO$ be a stratum that intersects $\calF$ but is not contained in it, then $\Gamma_x=\bbZ_q\subset S^1=SO(3)\cap SU(2)$ for every $x\in\calS$.
\end{cor}

We now prove Theorem \ref{ThmTop}.

\begin{proof}[Proof of Theorem \ref{ThmTop}]
Let $\calO$ be a compact orientable 4-dimensional orbifold with a simply connected underlying space and $n = \chi(|\calO|)$, then by Poincar\'e Duality and the Universal Coefficient Theorem, we conclude that
\[
H^k(|\calO|;\bbQ) = \begin{cases}
\bbQ,&k=0,4\\
0,&k=1,3\\
\bbQ^{n-2},&k=2.
\end{cases}
\]

Also, since $\pi_1(|\calO|)=0$, and by properties of CW-complexes, we can conclude that
\[
H^k(|\calO|;\bbZ) = \begin{cases}
\bbZ,&k=0\\
0,&k=1\\
\bbZ^{n-2},&k=2\\
\tau^3(|\calO|),&k=3\\
\bbZ+\tau^4(|\calO|),&k=4,
\end{cases}
\]
(where $\tau^k(X)$ represents that torsion in $H^k(X)$.) Our next step is to show that $\tau^4(|\calO|)=0$.

Let $p_1$ be a point in $\calO$, $B_\epsilon(p_1)$ be a small open ball around $p_1$. Define $X_1 = |\calO|\setminus B_\epsilon(p_1)$. From the structure of an orbifold, we know that $X_1$ deformation retracts onto a CW-complex with no cells of dimension greater than 3, and $\chi(X_1) = n-1$. Consider the Mayer-Vietoris sequence given by $|\calO| = X_1 \bigcup\limits_{\partial B_\epsilon(p_1)} \bar{B_\epsilon(p_1)}$. We see that $H^1(X_1) = 0$ and $H^2(X_1) = \bbZ^{n-2}$, so $H^3(X_1)$ is torsion only. So, we have a short exact sequence
\[
\tau^3(X_1) \to \bbZ \to \bbZ + \tau^4(|\calO|) \to 0.
\]
Since the only possible map from a torsion group to $\bbZ$ is trivial, we must have $\bbZ\to\bbZ+\tau^4(|\calO|)$ be an isomorphism, so $\tau^4(|\calO|)=0$.

Next we show that there exists a surjective map $\phi:\pi_1^{orb}(\calO)\to\tau^3(|\calO|)$.

Let $p_1,\ldots,p_k$ be all the points in $\calO$ whose neighborhoods are not topological disks (such points are isolated in 4-orbifolds). Let $\epsilon>0$ be small, and $|X| = |\calO| \setminus \left[\cup_{i=1}^n B_\epsilon(p_i)\right]$. Then, $|X|$ is a topological 4-manifold with boundary consisting of $n$ copies of finite quotients of $S^3$.

By topology of orbifolds, $\pi_1^{orb}(X) = \pi_1^{orb}(\calO)$, and $\pi_1^{orb}(X)\to \pi_1(|X|)$ is surjective. Furthermore, $\pi_1(|X|)\to H_1(|X|)$ is also surjective by Hurewicz, and $H_1(|X|)=H^3(|X|,|\partial X|)$ by the Poincar\'e duality for manifolds with boundary.

By excision we know that $H^k(|X|,|\partial X|) = H^k(|\calO|,\cup_{i=1}^n B_\epsilon(p_i))$ for all $k$, and from the long exact sequence for relative cohomology we know that $H^k(|\calO|,\cup_{i=1}^n B_\epsilon(p_i)) = H^k(|\calO|)$ for $k\geq 2$.

Therefore, we have a surjective composition $\pi_1^{orb}(\calO)\to\pi_1^{orb}(X)\to \pi_1(|X|) \to H_1(|X|) \to H^3(|X|,|\partial X|) \to H^3(|\calO|,\cup_{i=1}^n B_\epsilon(p_i)) \to H^3(|\calO|)$.
\end{proof}

\begin{rmk}
It is worth noting that most of this proof works equally well when $\pi_1(|\calO|)$ is finite rather than trivial, then we have 2 torsion groups $H^3(|\calO|)$ and $H_1(|\calO|)$ with $\pi_1^{orb}(\calO)$ surjecting onto both. If $\pi_1(|\calO|)$ is infinite, the author is unable to rule out the possibility that $H^4(|\calO|)$ has a non-trivial torsion component.
\end{rmk}

\begin{ques}
Let $n\geq 5$. If $\calO^n$ is a compact, orientable, $n$-dimensional orbifold with $\pi_1(|\calO|)=0$ does $\pi_1^{orb}(\calO)$ provide a bound on how $\bbZ$-valued Poincar\'e Duality for $|\calO|$ fails? In particular, is it true that $\pi_1^{orb}(\calO)=0$ implies that $\bbZ$-valued Poincar\'e Duality for $|\calO|$ holds?

It is worth noting that if $n=5$, then this question comes down to the question of how $\pi_1^{orb}(\calO)$ impacts $H^4(|\calO|)$ (which is torsion only).
\end{ques}

\section{Proof of Theorem \ref{ThmMain}}\label{SecPf}

We now prove Theorem \ref{ThmMain} in 2 parts. First we consider the case when the fixed point set consists of isolated fixed points, and then we consider the case where there exists a 2-dimensional component of the fixed point set.

\subsection{Isolated Fixed Point Case}

If $y$ is an isolated fixed point, then the lift of the slice representation is equivalent to
\[
\phi_{k,l}:S^1\times\bbC^2\to\bbC^2;\qquad
e^{i\theta}\star(z_1,z_2) = (e^{ik\theta/m}z_1,e^{il\theta/n}z_2),
\]
where $k,l\in\bbZ$ are relatively prime and $(e^{2\pi i/m},e^{2\pi i/n})\in \Gamma_y$. Furthermore, $\Gamma_y = \langle (e^{2\pi i/m},e^{2\pi i/n})\rangle \oplus \tilde{\Gamma_y}$. Let $S^3(1)\subset\bbC^2$ be the unit sphere and let $d:S^3(1)\times S^3(1)\to\bbR$ be given by $d(v,w)=\angle(v,w)$. Let $(X_{kl},d_{kl})$ be the orbit space of $S^3(1)/\tilde{S^1_{kl}}$ where $\tilde{S^1_{kl}}$ is a circle that acts by $(e^{ik\theta},e^{il\theta})$. Furthermore, let $(\tilde{X_{kl}},\tilde{d_{kl}})$ be the quotient of $X_{kl}$ by $\tilde{\Gamma_y}$.

\begin{lem}\label{LemAngle}
If $x_1,x_2,x_3\in \tilde{X_{kl}}$, then
\[
\tilde{d_{kl}}(x_1,x_2)+\tilde{d_{kl}}(x_2,x_3)+\tilde{d_{kl}}(x_3,x_1) \leq\pi
\]
\end{lem}

\begin{proof}
By Lemma 4 of \cite{HK4Dim}, this holds for $(X_{kl},d_{kl})$, take lifts of $x_i$'s, apply lemma 4, and then observe that $X_{kl}\to\tilde{X_{kl}}$ is distance non-increasing. Which gives the desired result.
\end{proof}

We now show that if $\calF$ consists of only isolated points, then it has at most 3 points.

Suppose that $\calF$ contains at least four points, call them $p_i, 1\leq i\leq4$. Let $l_{ij}=\dist(p_i,p_j)$ and let
\[
C_{ij}=\{\gamma:[0,l_{ij}] \to \calO|\gamma\;length\;minimizing\;p_i\;to\;p_j\}.
\]
For each triple $1\leq i,j,k\leq 4$ set
\[
\alpha_{ijk}=\min\{\angle(\gamma_j'(0),\gamma_k'(0))|\gamma_j\in C_{ij}, \gamma_k\in C_{ik}\}.
\]
Since $\calO$ is compact, the minimum exists.

By Toponogov theorem for orbifolds (see \cite{SOrbiIso}), we get that for $i,j,k$ distinct, $\alpha_{ijk}+\alpha_{kij}+\alpha_{jki}>\pi$. Summing over $i,j,k$, we get
\[
\sum_{i=1}^4\sum_{\substack{1\leq j < k\leq 4\\j,k\neq i}} \alpha_{ijk}>4\pi.
\]
On the other hand, by \ref{LemAngle}, we know that
\[
\sum_{\substack{1\leq j < k\leq 4\\j,k\neq i}} \alpha_{ijk}\leq\pi.
\]
Therefore, we can not have more than three isolated fixed points.

From Theorem \ref{ThmTop}, we conclude that if $\pi_1^{orb}(\calO)=0$, then $H^*(|\calO|)=H^*(S^4)$ or $H^*(\CP^2)$. The case where $H^*(|\calO|)=H^*(S^4)$ is of particular interest, since $\pi_1(|\calO|)=0$ allows us to use Hurewicz isomorphisms to see that $\pi_2(|\calO|)=\pi_3(|\calO|)=0$, and $\pi_4(|\calO|)=\bbZ$. Consider a map $\phi:S^4\to|\calO|$ that generates $\pi_4(|\calO|)$, then if one considers the long exact homotopy sequence, one can conclude that $\phi$ must be a homotopy equivalence.

\subsection{2-Dimensional Fixed Point Set Case}

For the case when $\dim\calF=2$, we utilize recent work of Harvey and Searle \cite{HSAlex} on isometries of Alexandrov spaces. In particular, we need the following:

\begin{thm}[\cite{HSAlex} Theorem C part (ii)]
Let a compact Lie group $G$ act isometrically and fixed-point homogeneously on $X^n$, a compact $n$-dimensional Alexandrov space of positive curvature and assume that $X^G\neq\emptyset$ and has a codimension 2 component, then:

The space $X$ is $G$-equivariantly homeomorphic to $(\nu\ast G)/G_p$, where $\nu$ is the space of normal directions to $G(p)$ where $G(p)$ is the unique orbit furthest from $F$.
\end{thm}

\begin{rmk}
The proof of this theorem comes from the work of Perelman on the soul conjecture and Sharafutdinov retraction for Alexandrov spaces. As well as the slice theorem.

Work of Perelman implies that the slice theorem extends to the codimension 2 fixed point set (which we will refer to as $N$). In the neighborhood of the soul orbit we have $G\times_{G_p} Cone(\nu)$ and in the neighborhood of $N$ we have $Cone(G)\times_{G_p}\nu$. Gluing the two components together we obtain a $G$-equivariant homeomorphism $|\calO| \cong (G\ast\nu)/G_p$.
\end{rmk}

Suppose we have $G_p=S^1$ at the soul point, and hence $\nu=S^3/\Gamma$. This implies that $|\calO| = (S^3/\Gamma\ast S^1)/S^1 = (S^5/\Gamma)/S^1 = \CP^2[\lambda]/\tilde{\Gamma}$. In particular, if $\pi_1^{orb}(\calO)=0$, then $|\calO|$ is homeomorphic to $\CP^2[\lambda]$.

Taking into account how $S^1$ must act on $\calO$, we conclude that $|\calO|=\CP^2[\lambda]/\bbZ_q$ where $\bbZ_q\subset T^2\subset\Isom(\CP^2[\lambda])$, $\bbZ_q$ fixes $[1:0:0]$, $S^1$-action lifts to an action on $\CP^2[\lambda]$ with $S^1\subset T^2$ fixing $[1:0:0]$ and $\{[0:z:w]\}$.

Next suppose that we have $G_p=\bbZ_q$ at the soul point, and hence $\nu=S^2$. This implies that $|\calO| = (S^2\ast S^1)/\bbZ_q = S^4/\bbZ_q$. In particular, if $\pi_1^{orb}(\calO)=0$, then $|\calO|$ is homeomorphic to $S^4$.

Once again, we use our partial knowledge of the $S^1$ action to conclude that $|\calO|=S^4/\bbZ_q$ where we view $S^4\subset\bbC^2\oplus\bbR$, and a generator of $\bbZ_q$ acts like $x\cdot(z,w;r)=(e^{2\pi i/q} z, e^{2\pi i k/q} w; r)$ where $(k,q)=1$. The $S^1$-action lifts to an action on $S^4$ given by $\theta\cdot(z,w;r)=(e^{i\theta} z, w; r)$.

\begin{conj}
Let $\calO$ be as in Theorem \ref{ThmMain}, then $|\calO|$ is homeomorphic to either $S^4$ or $|\CP^2[\lambda]|$.
\end{conj}

\begin{rmk}
The only case remaining is the when we only have isolated fixed points.

The author has been able to show that if $H^*(|\calO|)=H^*(S^4)$, then the neighborhoods of the two fixed points are homeomorphic to either $B^4$ or a cone over the Poincar\'e Dodecahedral space. Furthermore, using the fact that an orbifold has a natural PL structure, one can show that if one of the neighborhoods is $B^4$ then so is the other.

If one is able to rule out the case where both points have neighborhoods homeomorphic to a cone over the Poincar\'e Dodecahedral space, then one can conclude that $|\calO|$ is homeomorphic to $S^4$.
\end{rmk}

\bibliographystyle{amsalpha}
\bibliography{References}
\end{document}